\documentclass[11pt,a4paper,fleqn]{scrartcl}    
\usepackage[english]{babel}                     
\usepackage[T1]{fontenc}                   
\usepackage[applemac]{inputenc}
\usepackage{graphicx}                       %
\usepackage[font=small]{quoting}      
\usepackage{enumitem,wrapfig}
\usepackage[colorlinks=true,linkcolor=blue,citecolor=blue]{hyperref}
\usepackage{xcolor}
\usepackage{pgf,tikz}
\usepackage{picture}
\usepackage{amsmath,amsthm,amsfonts,amssymb}

\newcommand{\h}{d\mathcal{H}^{n-1}}

\newtheorem{theorem}{Theorem}[section]
\newtheorem{lemma}{Lemma}[section]
\newtheorem{prop}{Proposition}[section]

\theoremstyle{definition}

\newtheorem{rem}{Remark}[section]

\newcommand{\ds}{\displaystyle}

\newcommand{\R}{\mathbb R}

\newcommand{\de}{\partial}
\newcommand{\eps}{\varepsilon}

\begin{document}
\title{An optimal insulation problem}
 \author{Francesco Della Pietra$^{*}$, \\ Carlo Nitsch$^{*}$,\\Cristina Trombetti%
       \thanks{Universit\`a degli studi di Napoli Federico II, Dipartimento di Matematica e Applicazioni ``R. Caccioppoli'', Via Cintia, Monte S. Angelo - 80126 Napoli, Italia.
       Email: f.dellapietra@unina.it, carlo.nitsch@unina.it, cristina.trombetti@unina.it}
}
\date{}
\maketitle
%

\begin{abstract}
\noindent{\textsc{Abstract.}} 
In this paper we consider a minimization problem which arises from thermal insulation. 
A compact connected set $K$, which represents a conductor of constant temperature, say $1$, is thermally insulated by surrounding it with a layer of thermal insulator, the open set $\Omega\setminus K$ with $K\subset\bar\Omega$. The heat dispersion is then obtained as
\[
\inf\left\{ \int_{\Omega}|\nabla \varphi|^{2}dx +\beta\int_{\de^{*}\Omega}\varphi^{2}d\mathcal H^{n-1} ,\;\varphi\in H^{1}(\R^{n}), \, \varphi\ge 1\text{ in } K\right\},
\]
for some positive constant $\beta$.\\
We mostly restrict our analysis to the case of an insulating layer of constant thickness. We let the set $K$ vary, under prescribed geometrical constraints, and we look for the best (or worst) geometry in terms of heat dispersion. We show that under perimeter constraint the disk in two dimensions is the worst one. The same is true for the ball in higher dimension but under different constraints. We finally discuss few open problems.

\end{abstract}

\section{Introduction}

As energy saving and noise pollution are growing in importance year after year, insulation represents one of biggest challenge for environmental improvement. Even if insulation is one of the oldest and most studied problem in mathematical physics, the mathematics involved is still very tricky especially when one looks at shape optimization questions. In this paper we will consider a domain of given temperature, thermally insulated by surrounding it with a constant thickness of thermal insulator. As the geometry of the set varies (under prescribed geometrical constraints to be specified later) it is reasonable to ask for the best (or worst) choice (in terms of heat dispersion).

Let $K$ be a compact set in $\R^n$ and let $\Omega$ be a bounded connected open set of $\R^{n}$, with $K\subset\bar\Omega$ (here the set $K$ represents a thermally conducting body, surrounded by an amount of insulating material $\Omega\setminus K$). The heat dispersion is given by
\begin{equation}
\label{dispintro}
I_{\beta}(K;\Omega)=\inf\left\{ \int_{\Omega}|\nabla \varphi|^{2}dx +\beta\int_{\de^{*}\Omega}\varphi^{2}d\mathcal H^{n-1} ,\;\varphi\in H^{1}(\R^{n}), \, \varphi\ge 1\text{ in } K\right\}.
\end{equation}
where $\beta>0$ is a fixed parameter depending on the physical characteristics of the problem, and $\de^{*}\Omega$ stands for the reduced boundary of $\Omega$. 


If $\Omega$ has Lipschitz boundary and $K\subset\Omega$, then there exists a minimiser $u$ of \eqref{dispintro} which solves in the weak sense the Robin-Dirichlet problem
\begin{equation}
\label{elinsintro}
\begin{cases}
\Delta u=0 &\text{in }\Omega\setminus K,\\
u=1 &\text{in } K,\\
\frac{\de u}{\de \nu} +\beta u=0 &\text{on }\de \Omega,
\end{cases}
\end{equation}
and it holds that
\[
I_\beta(K,\Omega)=\beta\int_{\de\Omega}u \, dx.
\]
In what follows we will use $D$ to denote a bounded open connected subset of $\Omega$ and in that case the notation $I_\beta(D,\Omega)$ stands for $I_\beta(\bar D,\Omega)$.

If $D=\Omega$ (no insulation),  we set
\[
I_{\beta}(D;D)=\beta P(D),
\]
where $P(D)$ stands for the classical perimeter in $\R^n$.
One of the most intriguing questions in thermal insulation is to study the configurations of $D$ and $\Omega$ for which $I_\beta(D,\Omega)$ is minimal or maximal, under reasonable geometric constraints.

Shape optimization problems of this or similar nature already appeared in several papers, as for instance in \cite{acbu86,brca80,bubu05,bbn1,bbn2,bu88,dpnst,d,er,f}. Related results are contained in \cite{dpp}.

{  In Section \ref{section_set} we consider $I_\beta(D,\Omega)$ when $\Omega$ is the Minkowski sum
\[
\Omega=D+\delta B
\]
(see Section 2 for its definition),
 $B$ is the unit ball centered at the origin, and $\delta$ is a positive number. 
 This corresponds to accomplish the insulation of $D$ by adding an insulation layer of constant thickness. 
 For the sake of simplicity, we set
\[
I_{\beta,\delta}(D)\equiv I_\beta(D,D+\delta B).
\]

In Section \ref{max2} we consider the planar case and prove that in the class of smooth domains $D$ with given perimeter $P$, and fixed $\delta>0$, the disk maximizes $I_{\beta,\delta}$. 
Therefore, we deduce that the worst possible geometry (in terms of heat dispersion) is the symmetric one.}

Thereafter,  in Section \ref{maxn}, we prove that in higher dimension ($n\ge 3$)  balls still maximize $I_{\beta,\delta}$, but our result finds its natural generalization in the class of convex domains $D$ with given $(n-1)-$quermassintegral. 

The study of the worst possible shape, the one with the highest heat dispersion, is justified by the idea that, in designing the optimal insulation, one has an a-priori bound of the heat dispersion in terms of geometric quantities alone.

However, it is of great interest as well, to study the existence of the best geometries: those which minimize heat dispersion.
In the last section we make some remarks on that, and discuss few open problems. For instance we consider the minimization of $I_{\beta,\delta}(D)$ in the class of sets of given perimeter or measure. Comparison with classical Capacitary problem (which corresponds to the case $\beta\to\infty$) suggests that the minimization of 
$I_{\beta,\delta}(D)$ when we vary $D$ by preserving its measure, occurs when $D$ is a ball.
Even more complicated is the case where both $\Omega$ and $D$ vary among sets of fixed volume. However such a case requires a careful formulation otherwise the problem can be ill-posed, in the sense that (against common sense) sometimes insulation might increase heat dispersion. In particular we show that, when $D=B_R$ (open ball of radius $R$), and $\beta$ is small enough, 
there exists a positive constant $\delta_{0}$ (which depends on $\beta$ and $R$ alone) such that for any bounded connected open set $\Omega$, with $\Omega\supset B_R$ and $|\Omega \setminus B_R|<\delta_{0}$, then 
\[
I_{\beta}(B_R;\Omega)>I_{\beta}(B_R;B_R).
\]

\section{Insulating problem}\label{section_set}
Given an open, bounded, connected  set  $D\subset \R^{n}$, and $\Omega=D+\delta B$, $\delta>0$, where $B$ is the unit ball centered at the origin and $D+\delta B$ stands for the Minkowski sum of the two sets, that is
\[
\Omega=D+\delta B=\{x+\delta y,\; x\in \Omega,\, y\in B\},
\]
we are interested in the study of the properties of
\begin{equation}
\label{disp}
I_{\beta,\delta}(D)=\inf\left\{ \int_{\Omega}|\nabla \varphi|^{2}dx +\beta\int_{\de^{*}\Omega}\varphi^{2}d\mathcal H^{n-1} ,\;\varphi\in H^{1}(\R^{n}), \, \varphi\ge 1\text{ in } \bar D\right\}
\end{equation}
with $\beta>0$. 

We denote by $W_{\bar D}^{1,2}(\Omega)$  the closure in $W^{1,2}(\Omega)$ of $\{ \left.\varphi\right|_\Omega :\varphi  \in {\mathcal C}_c^{\infty }(\R^{n}) \> \hbox{with} \,\bar D \cap \hbox {supp} \varphi = \emptyset \} $.
If $\Omega$ has Lipschitz boundary, then there exists a unique minimizer $u$ of \eqref{disp} such that $u-1 \in W_{\bar D}^{1,2}(\Omega)$ which solves
\begin{equation}
\label{elins}
\begin{cases}
\Delta u=0 &\text{in }\Omega\setminus \bar D,\\
u=1 &\text{in } \bar D,\\
\frac{\de u}{\de \nu} +\beta u=0 &\text{on }\de \Omega.
\end{cases}
\end{equation}
in the sense that
\begin{equation}
\int_{\Omega\setminus \bar D} \nabla u \nabla \varphi \, dx + \beta\int_{\de \Omega}u \, \varphi\,d\mathcal H^{n-1}=0
\end{equation}
for all $ \varphi \in W_{\bar D}^{1,2}(\Omega)$.

Then
\begin{equation}
I_{\beta,\delta}(D)=\beta\int_{\de\Omega}u d\mathcal H^{n-1}.
\end{equation}
We will call $I_{\beta,\delta}(D)$ the ``heat dispersion'' and we  observe that if $\beta=0$, the condition on $\de \Omega$ in \eqref{elins} becomes a Neumann condition: $\frac{\de u}{\de \nu}=0$, while if $\beta=+\infty$, the problem has to be meant with a Dirichlet condition on $\de\Omega$: $u=0$. 
\begin{rem}
It holds that
\[
I_{\beta,\delta}(D)\le \textrm{cap}_{2}({ \bar D},\Omega)=\inf
\left\{ \int_{\Omega} |\nabla \varphi|^{2}dx,\;\varphi\in H_{0}^{1}(\Omega), \, \varphi\ge 1\text{ in }\bar D\right\}.
\]
Furthermore, by choosing $\varphi=1$ in \eqref{disp} as test function, we get
\[
I_{\beta,\delta}(D)\le \beta P(\Omega),
\]
where $P(\Omega)$ stands for the perimeter of $\Omega$ in $\R^n$. 
\end{rem}

\section{Maximization of the heat dispersion: the planar case}
\label{max2}
In this section we consider the two dimensional case. For an open, bounded,  connected set $D$ we denote by $D^{*}$ the disk having the same perimeter of $D$. We consider $\Omega=D+\delta B$ and the disk $\Omega_*=D^{*}+\delta_{*} B$, where $B$ is the unit disk centered at the origin.

If $D$ is convex, the Steiner formulae state that
\[
\begin{array}{ll}
|\Omega|=|D|+P(D)\delta +\pi\delta^{2},& P(\Omega)=P(D)+2\pi \delta,
\\[.2cm]
|\Omega_*|=|D^{*}|+P(D^{*})\delta_{*} +\pi\delta_{*}^{2},  & P(\Omega_*)=P(D^{*})+2\pi \delta_{*},
\end{array}
\]
If we ask that the area of the insulating material $\Omega\setminus \bar D$ remains constant, then
\[
|\Omega|-|D|=P(D)\delta +\pi \delta^{2}=|\Omega_*|-|D^{*}|= P(D^{*})\delta_{*}+\pi \delta_{*}^{2}.
\]
Then, since $P(D)=P(D^{*})$ then $\delta =\delta_{*}$ and, as byproduct, $P(\Omega)=P(\Omega_*)$. On the contrary, if $\delta =\delta^{*}$, then $|\Omega|-|D|=|\Omega_{*}|-|D^{*}|$.

For a general bounded domain with piecewise $C^{1}$ boundary, it holds that
\begin{equation}
\label{steingen}
|\Omega|\le |D|+P(D)\delta +\pi\delta^{2},\quad P(\Omega)\le P(D)+2\pi \delta,
\end{equation}
hence $\delta=\delta_{*}$ implies 
\[
|\Omega|-|D|\le |\Omega_*|-|D^{*}|,
\]
which means that the area of the insulating material increases keeping fixed the perimeter $P(D)$ and the thickness $\delta$.

\begin{theorem}
Let $D$ be an open, bounded, connected set  of $\R^{2}$ with piecewise $C^{1}$ boundary. Then 
\[
I_{\beta,\delta}(D)\le I_{\beta,\delta}(D^{*}),
\] 
where $D^{*}$ is the disk having the same perimeter of $D$, that is $P(D^{*})=P(D)$. 
\end{theorem}
\begin{proof}
Let $v$ be the radial minimizer of $I_{\beta,\delta}(D^{*})$.
 We denote by $R$ be the radius of $D^{*}$, $\Omega_{*}=D^{*}+\delta B$, $v_m=v(R+\delta) = \min_{\Omega_{*}} v$ and by $  \max_{\Omega_*} v=v(R)=1$. Being $v$ radial, the modulus of the gradient of $v$ is constant on the level lines of $v$.

Let us consider the function 
\[
g(t)=|Dv|_{v=t},\; v_m<t \le 1.
\]

 Let $d(x)$ be the distance of a point $x$ from $D$ and set  
 \[
 w(x)=G\left(R+d(x)\right),\quad  x \in\Omega,\qquad \text{where } G^{-1}(t)=R+\displaystyle \int_{t}^{1} {\frac{1}{g(s)}}ds.
 \]
 By construction $w \in H^1(\Omega)$ and, being $G$ decreasing, it results:
 
 \begin{equation}
 \label{testfunctions}
 \begin{array}{ll}
 &   \max_{\Omega} w = \left.w\right|_{\de D}= 1= G(R); \cr \\
 &  w_m= \min_{\Omega} w = \left.w\right|_{\de \Omega} =G(R+\delta) = v_m  \cr  \\
 &  |Dw|_{w=t} = |Dv|_{v=t} = g(t) \quad w_m \le t \le 1.
 \end{array}
 \end{equation}
Hence $w$ is a test function. Then
\[
I_{\beta,\delta}(D)\le \int_{\Omega\setminus \bar D}|\nabla w|^{2}dx +\beta\int_{\de\Omega}w^{2}d\mathcal H^{1}
\]
Let 
\[
E_t = \{ x \in \Omega : w(x) >t\} = \{ x \in \Omega : d(x) < G^{-1}(t)\}=D+G^{-1}(t)B,
\]
and let
\[
B_t = \{ x \in \Omega^\star: v(x) >t\}. 
\]  
\noindent  By Steiner formula \eqref{steingen} we get 
\begin{equation}\label{perimetris}
P(E_{t}) \le P(D)+2\pi G^{-1}(t) =P(D^{*})+2\pi G^{-1}(t)=2\pi(R+G^{-1}(t))=P(B_{t})
\end{equation}
for every $w_m < t \le 1$.
Hence, by \eqref{perimetris},
\[
\int_{w=t}|Dw|\,d\mathcal H^{1} =g(t) P(E_t)\le g(t) P(B_t)= \int_{v=t}|Dv| \,d\mathcal H^{1}, \qquad w_m < t \le 1
\]
then, by co-area formula and \eqref{testfunctions},
\begin{eqnarray*}\label{energie}
\int_{\Omega\setminus \bar D} |Dw|^2\,dx&=&\int_{w_m}^{1} g(t)P(E_t)\,dt
\\
& \le & \int_{w_m}^{1} g(t)P(B_t)\,dt = \int_{v_m}^{1} g(t)P(B_t)\,dt 
\\
&=& \int_{\Omega_{*}\setminus \bar D^{*}}|Dv|^2\,dx. \notag
\end{eqnarray*}

As regards the boundary terms, since by construction $w = w_{m}=v_m$ on $ \de \Omega$ (see \eqref{testfunctions}), and $P(\Omega) = P(\Omega_*)$, we have
\begin{equation*}
\label{boundary}
\int_{\partial \Omega} w^2 \, d\mathcal H^{1} = w_{m}^{2} P(\Omega) = v_m^2 P(\Omega_{*}) = \int_{\partial \Omega_*} v^2 \,d \mathcal H^{1}
\end{equation*}
Hence
\begin{multline*}
I_{\beta,\delta}(D)\le \int_{\Omega\setminus D}|\nabla w|^{2}dx +\beta\int_{\de\Omega}w^{2}d\mathcal H^{1} 
\\ = 
\int_{\Omega_*\setminus D^{*}}|\nabla v|^{2}dx +\beta\int_{\de\Omega_*}v^{2}d\mathcal H^{1} =
I_{\beta,\delta}(D^{*}).
\end{multline*}
\end{proof}
%
%
%

\section{The $n$-dimensional case}
\label{maxn}
\subsection{Preliminaries on convex sets}
Here we list some basic facts on convex sets. All the notions and results can be found, for example, in \cite{bz,schn}.
Let $K$ be a nonempty, bounded, convex set in $\R^{n}$ and let $\delta>0$. Then the Steiner formulas for the volume and the perimeter  read  as
\begin{eqnarray}
\label{vol}
|K+\delta B|&=&\sum_{j=0}^{n }\binom{n }{j}W_{j}(K)\delta^j
\\
&=& |K| + n W_{1}(K)\delta+\frac{n(n-1)}{2}W_2(K)\delta^{2} + ... +\omega_{n}\delta^n. \notag \\
\label{per}
P(K+\delta B)&=&n\sum_{j=0}^{n-1}\binom{n-1}{j}W_{j+1}(K)\delta^j
\\
&=& P(K)+n(n-1)W_2(K)\delta + ... +n\omega_{n}\delta^{n-1}, \notag
\end{eqnarray}
where $B$ is the unit ball in $\R^{n}$ centered at the origin, whose measure is denoted by $\omega_{n}$. The coefficients $W_{j}(K)$ are the so-called quermassintegrals of $K$. 

It immediately follows that
\begin{equation}\label{derivata}
\lim_{\delta \to 0^+}\frac{P(K+\delta B)-P(K)}{\delta}=n(n-1)W_2(K).
\end{equation}
If $K$ as $C^2$ boundary, with nonzero Gaussian curvature, the quermassintegrals are related to the principal curvatures of $\de K$. Indeed, in such a case
\begin{equation}
\label{quersmooth}
W_i(K)=\frac 1 n \int_{\de K} H_{i-1}(x) d \mathcal H^{n-1}, \quad
i={1,\ldots n}.
\end{equation}
Here $H_j$ denotes the $j-$th normalized elementary symmetric function of the principal curvatures of $\de K$,
that is $H_0=1$ and
\[
H_j(x)= \binom{n-1}{j}^{-1} \sum_{1\le i_1\le \ldots \le i_j\le n-1}
\kappa_{i_1}(x)\cdots \kappa_{i_j}(x),\quad j={1,\ldots,n-1},
\]
where $\kappa_1(x),...,\kappa_{n-1}(x)$ are the principal curvatures at a point $x\in \de K$.
In particular, by \eqref{derivata} and \eqref{quersmooth} we get also that
\begin{equation*}
\lim_{\delta \to 0^+}\frac{P(K+\delta B)-P(K)}{\delta}=(n-1)\int_{\partial K} H_{1}(x) \h,
\end{equation*}
where $H_{1}(x)$ is the mean curvature of $\partial K$ at a point $x$.

The Aleksandrov-Fenchel inequalities state that
\begin{equation}
  \label{afineq}
\left( \frac{W_j(K)}{\omega_n} \right)^{\frac{1}{n-j}} \ge \left(
  \frac{W_i(K)}{\omega_n} \right)^{\frac{1}{n-i}}, \quad 0\le i < j
\le n-1,
\end{equation}
where the inequality is replaced by an equality if and only if $K$ is a ball. 

In what follows, we use the Aleksandrov-Fenchel inequalities for
particular values of $i$ and $j$.
When $i=0$ and $j=1$, we have the classical isoperimetric inequality:
\[
P(K) \ge n \omega_n^{\frac 1 n} |K|^{1-\frac 1 n}.
\]
Moreover, if $i=k-1$, and $j=k$, we have
\[
W_k(K) \ge \omega_n^{\frac{1}{n-k+1}} W_{k-1}(K)^{\frac{n-k}{n-k+1}}.
\]

Let us denote by $K^{*}$ a ball such that $W_{n-1}(K)=W_{n-1}(K^{*})$. Then by Aleksandrov-Fenchel inequalities \eqref{afineq}, for $0\le i < n-1$
\[
\left(\frac{W_{i}(K^{*})}{\omega_{n}}\right)^{\frac{1}{n-i}}= \frac{W_{n-1}(K^{*})}{\omega_{n}}=\frac{W_{n-1}(K)}{\omega_{n}}
\ge \left(\frac{W_{i}(K)}{\omega_{n}}\right)^{\frac{1}{n-i}}.
\]
hence 
\begin{equation}
\label{afapp}
W_{i}(K)\le W_{i}(K^{*}),\quad 0\le i\le n-1
\end{equation}

\subsection{Maximization of the heat dispersion in convex domains}


\begin{theorem}
Let $D$ be an open, bounded, convex set  of $\R^{n}$. Then 
\[
I_{\beta,\delta}(D)\le I_{\beta,\delta}(D^{*}),
\] 
where $D^{*}$ is the ball having the same $W_{n-1}$ quermassintegral of $D$, that is $W_{n-1}(D)=W_{n-1}(D^{*})$. 
\end{theorem}
\begin{proof}
Let be $\Omega_{*}=D^{*}+\delta B$, and $v$ the radial minimizer of $I_{\beta,\delta}(D^{*})$.
 Since $\Omega = D+ \delta B$, Steiner formula \eqref{per} and  \eqref{afapp} imply 
$P(\Omega) \leq P(\Omega_*)$.

We denote by $v_m=v(R+\delta) = \min_{\Omega_{*}} v$ and by $  \max_{\Omega_{*}} v=v(R)=1$. Being $v$ radial, the modulus of the gradient of $v$ is constant on the level lines of $v$.

Let us consider the function 
\[
g(t)=|Dv|_{v=t},\quad v_m<t \le 1.
\]

 Let $d(x)$ be the distance of a point $x\in \Omega$ from $D$ and set  
 \[
 w(x)=G\left(R+d(x)\right),\quad  x \in\Omega,\qquad \text{where } G^{-1}(t)=R+\displaystyle \int_{t}^{1} {\frac{1}{g(s)}}ds.
 \]
 By construction $w \in H^1(\Omega)$ and being $G$ decreasing it results:
 
 \begin{equation}
 \label{testfunction}
 \begin{array}{ll}
 &   \max_{\Omega} w = \left.w\right|_{\de D}= 1= G(R); \cr \\
 &  w_m= \min_{\Omega} w = \left.w\right|_{\de \Omega} =G(R+\delta) = v_m \qquad \cr  \\
 &  |Dw|_{w=t} = |Dv|_{v=t} = g(t) \quad w_m \le t \le 1,
 \end{array}
 \end{equation}
Then
\[
I_{\beta,\delta}(D)\le \int_{\Omega\setminus D}|\nabla w|^{2}dx +\beta\int_{\de\Omega}w^{2}d\mathcal H^{n-1}
\]
Let 
\[
E_t = \{ x \in \Omega : w(x) >t\} = \{ x \in \Omega : d(x) < G^{-1}(t)\}=D+G^{-1}(t)B,
\]
and
\[
B_t = D\{ x \in \Omega_{*}: v(x) >t\}. 
\]  
Being $W_{n-1}(D)=W_{n-1}(D^{*})$, using the Steiner formula and  \eqref{afapp}, we get for $w_m < t \le 1$ and $\rho=G^{-1}(t)$
\begin{align*}
P(E_{t})=
P(D+\rho B)&=n\sum_{k=0}^{n-1}\binom{n-1}{k}W_{k+1}(D)\rho^k\\
&\le n\sum_{k=0}^{n-1}\binom{n-1}{k}W_{k+1}(D^{*})\rho^k = P(D^{*}+\rho B)=P(B_{t}).
\end{align*}
Hence
\[
\int_{w=t}|Dw|\,\h =g(t) P(E_t)\le g(t) P(B_t)= \int_{v=t}|Dv| \,\h, \qquad w_m < t \le 1
\]
then, by co-area formula and \eqref{testfunction},
\begin{eqnarray*}\label{energie}
\int_{\Omega\setminus D} |Dw|^2\,dx&=&\int_{w_m}^{1} g(t)P(E_t)\,dt
\\
& \le& \int_{w_m}^{1} g(t)P(B_t)\,dt \le \int_{v_m}^{1} g(t)P(B_t)\,dt 
\\
&=& \int_{\Omega_{*}\setminus D^{*}}|Dv|^2\,dx. \notag
\end{eqnarray*}

As regards the boundary behavior, since by construction $w = w_{m}=v_m$ on $ \de \Omega$ (see \eqref{testfunction}), and $P(\Omega) \leq P(\Omega_*)$, we have
\begin{equation*}
\int_{\partial \Omega} w^2 \, \h = w_{m}^{2} P(\Omega) \le v_m^2 P(\Omega_{*}) = \int_{\partial \Omega_{*}} v^2 \, \h.
\end{equation*}
%
%
%
\end{proof}
\begin{rem}
Let us observe that whenever $W_{n-1}(D)=W_{n-1}(D^{*})$, then by the Steiner formula and the Aleksandrov-Fenchel inequalities then $\left|\Omega\right|-\left|D\right| \le \left|\Omega_{*}\right|-\left|D^{*}\right|$.
\end{rem}



\section{Remarks and Open Problems}
As regards the infimum of $I_{\beta,\delta}(D)$ among sets with fixed perimeter, it is easy to show that, without any geometrical restriction, such an infimum is zero. For example, for $n=2$, even limiting the analysis to open connected sets is not enough to bound the infimum away from zero. To this aim we consider the sequence of sets $Q_{k}=\left]0,\frac{1}{k}\right[^{2}$. Then we construct the set $D_{k}$ by removing from $Q_{k}$ $k^{2}$ disjoint closed squares of side $\frac{1}{4k^{2}}-\frac{1}{k^{3}}$ each ($k\ge 5$). Then for $k$ sufficiently large, it holds that $I_{\delta,\beta}(D_{k})=I_{\delta,\beta}(Q_{k})$. Hence
\[
I_{\beta,\delta}(D_{k})\to 0\quad\text{ as }k\to +\infty,
\] 
while for any $k$
\[
P(D_{k})= 1.
\]
The class of simply connected sets suffers the same problem. 
However, among planar convex sets of given perimeter, any minimising sequence for $I_{\beta,\delta}(D)$ admits, by Blaschke-Santal\'o Theorem, a subsequence which converges (possibly degenerating into a segment) to some convex body of same perimeter (to be understood as twice its length, or Minkowski content, in case of a segment). Since the functional $I_{\beta,\delta}(\cdot)$ is continuous under converging sequence of convex bodies, a minimum is achieved. This leads to the following

\vskip .2cm{\textbf{Open Problem 1.}} \textit{Find the minimum of $I_{\beta,\delta}(\cdot)$ in the class of convex planar sets of given perimeter.}\vskip .2cm

Another closely related question is 

\vskip .2cm{\textbf{Open Problem 2.}} \textit{Prove or disprove in any dimension that the minimum of $I_{\beta,\delta}(\cdot)$ among sets of given volume exists.}\vskip .2cm

For $\beta\to\infty$ the functional $I_{\beta,\delta}(D)$ converges to the capacity of the set $\bar D$ with respect to $D+\delta B$. For the capacity one can employ standard techniques, like the Schwarz symmetrization, to deduce that a minimum is achieved on balls. This suggests that, at least for large $\beta$, balls might provide a positive answer to Open Problem 2. 

Yet another important question is 

\vskip .2cm{\textbf{Open Problem 3.}} \textit{Prove or disprove in any dimension that for positive constants $\beta$ and $0<m<M$, there exists a minimum of
$$\{I_{\beta}(D;\Omega): D\subseteq\Omega,\, |D|=m,\, |\Omega|\le M \}.$$}\vskip .2cm
 
Again, comparison with Capacitary problem and common sense suggest that the minimum is achieved when $\Omega$ and $D$ are concentric balls.
The fact that one can not prescribe the exact volume of $\Omega$ but only an upper bound is due to the some counterintuitive behaviour
of the functional $I_{\beta}(D;\Omega)$ which is very peculiar and that can be summarised in next two Propositions. 

\begin{prop}
If $\beta\ge \frac{n-1}{R}$ then $I_{\beta,\delta}(B_{R})$ is decreasing in $\delta$. When $\beta<\frac{n-1}{R}$ 
then $I_{\beta,\delta}(B_{R})$ is increasing when $\delta\le\frac{n-1}{\beta}-R$ and decreasing when $\delta\ge\frac{n-1}{\beta}-R$
(see also \cite[\S 3.3.1--3.3.2]{IDBL}).
\end{prop}
\begin{proof}
If $D=B_{R}$ is a ball with radius $R$, then obviously $\Omega=B_{R}+\delta B=B_{R+\delta}$, and the minimum of \eqref{disp}, $v(x)=v(r)$, is
\[
v(r)=
\begin{cases}
1-c_{2}(\beta,\delta)\log\dfrac{r}{R}, &\text{if }n=2\\[.3cm]
1-c_{n}(\beta,\delta)\left(1-\dfrac{R^{n-2}}{r^{n-2}}\right) &\text{if }n>2,
\end{cases}
\qquad r\in[R,R+\delta]
\]
with
\[
\left\{
\begin{array}{lr}
c_{2}(\beta,\delta)=\ds \frac{\beta}{\frac{1}{R+\delta}+\beta \log\frac{R+\delta}{R}},& \\[.5cm]
c_{n}(\beta,\delta)=\ds \frac{\beta}{(n-2)\frac{R^{n-2}}{(R+\delta)^{n-1}}+ \beta\left(1-\frac{R^{n-2}}{(R+\delta)^{n-2}}\right)} &\;\;\;\;(n>2).
\end{array}\right.
\]
In particular, if $n=2$, computing the heat dispersion $I_{\beta,\delta}(B_{R})$, we have
\begin{equation*}
I_{\beta,\delta}(B_{R})=\int_{\de B_{R}} \frac{\de u}{\de \nu}d\mathcal H^{1}= 2\pi c_{2}(\beta,\delta).
\end{equation*}
We recall that $\nu$ is the outer normal to $B_{R+\delta}\setminus \bar B_{R}$. 

Let us observe that
\[
\left.I_{\beta,\delta}(B_{R})\right|_{\beta=0}=0,\quad \left.I_{\beta,\delta}(B_{R})\right|_{\beta=+\infty}=\frac{2\pi}{\log\left(1+\frac{\delta}{R}\right)}.
\]
We have
\[
\de_{\delta} \left[c_{2}(\beta,\delta) \right]\left[\frac{1}{R+\delta}+\beta \log\frac{R+\delta}{R}\right] =\frac{c_{2}(\beta,\delta)\beta}{(R+\delta)^{2}}
\left(\frac1\beta- R-\delta\right)
\]
that is
 
\[
\de_{\delta} \left[ I_{\beta,\delta}(B_{R})\right] < 0 \iff \delta > \frac{1}{\beta}-R. 
\]
A similar computation provides the bound for $n>2$.
\end{proof}

Therefore in the regime $\beta$ ``small'', insulation increases the heat dispersion if the insulator thickness is below a certain threshold value. A more careful analysis brings to a sharper result.

\begin{prop}
\label{isozero}
Let $D=B_{R}(0)$ and $\beta<\frac{n-1}{R}$. Then there exists a positive constant $\delta_{0}$ such that for any bounded domain $\Omega$, with $D\subset \Omega$ and $|\Omega|-|D|<\delta_{0}$, then 
\[
I_{\beta}(B_{R};\Omega)>I_{\beta}(B_{R};B_{R}).
\]
\end{prop}
To prove this result, we first need the subsequent Lemma \ref{perimetribound}. We introduce first the following notation. 

Let $D=B_{R}\subset \Omega$, and denote by 
\[
P=P(B_{R})=n\omega_{n}R^{n-1},\; V=|B_{R}|=\omega_{n}R^{n},\; \Delta P = P(\Omega) - P(B_{R}),\; \Delta V= |\Omega|-|B_{R}|.
\]
We first need the following Lemma.
\begin{lemma}
\label{perimetribound}
Let be $D=B_{R}$, and $B_{R}\subset\Omega$. For any $\delta_{0}>0$, there exists a constant 
\begin{equation}
\label{costante}
C=\frac{n\omega_{n}R^{n-1}}{\delta_{0}}\left[\left(1+\frac{\delta_{0}}{\omega_{n}R^{n}}\right)^{1-\frac{1}{n}}-1\right]
\end{equation}
such that if $\Delta V\le \delta_{0}$ it holds that
\[
\Delta P \ge C \Delta V.
\]
\end{lemma}
\begin{proof}
By the isoperimetric inequality it holds that
\[
n^{\frac{n}{n-1}}\omega_{n}^{\frac{1}{n-1}} = \frac{P^{\frac{n}{n-1}}}{V} \le \frac{(P+\Delta P)^{\frac{n}{n-1}}}{V+\Delta V}.
\]
Hence
\[
\Delta P \ge 
n \omega_{n}^{\frac{1}{n}}\left(\omega_{n} R^{n}+\Delta V\right)^{1-\frac{1}{n}}-n \omega_{n} R^{n-1}
\]
and then, if $\Delta V\le \delta_{0}$, we have that
\[
\frac{\Delta P}{\Delta V} \ge \frac{n \omega_{n}R^{n-1}}{\Delta V}\left[\left(1+\frac{\Delta V}{\omega_{n}R^{n}}\right)^{1-\frac{1}{n}}-1\right] \ge C,
\]
where $C$ is the constant in \eqref{costante}, and this complete the proof.
\end{proof}

\begin{proof}[Proof of Theorem \ref{isozero}]


Let $u$ be the minimizer of $I_{\beta}(B_{R};\Omega)$.  
Let us consider
\[
\Sigma=\Omega\setminus B_{R},\quad \Gamma_{m}=\de \Omega\setminus \de B_{R}, \quad \Gamma_{t}=\de \{u>t\} \setminus \de B_{R}, \quad \Gamma_{1}= \de B_{R}\cap \Omega,
\]
and
\[
p(t)=P(\{u>t\}\cap \Sigma),\qquad\text{for a.e. }t>0.
\]
We want to show that
\[
I_{\beta}(B_{R};B_{R})= \beta P(B_{R}) < I_{\beta}(B_{R};\Omega)=\int_{\Omega} |\nabla u|^{2}dx+\beta\int_{\de \Omega} u^{2}d\mathcal H^{n-1},
\]
or equivalently
\begin{equation}
\label{fine}
\mathcal H^{n-1}(\Gamma_{1}) < \frac{1}{\beta}\int_{\Omega} |\nabla u|^{2}dx + \int_{\Gamma_{0}} u^{2}d\mathcal H^{n-1},
\end{equation}
Then, using  coarea formula and Fubini theorem we have
\begin{align}
\notag \int_{0}^{1} t p(t) dt &= 
\int_{0}^{1} t \mathcal H^{n-1}(\Gamma_{1}) dt + \int_{0}^{1} t \mathcal H^{n-1}(\Gamma_{t}\cap \Omega) dt + \int_{0}^{1} t \mathcal H^{n-1}(\Gamma_{t}\cap \de\Omega) dt \\[.3cm]
&=\frac{\mathcal H^{n-1}(\Gamma_{1})}{2}+ \int_{\Omega} u |\nabla u|\,dx + \frac{1}{2} \int_{\Gamma_{0}} u^{2}\, d\mathcal H^{n-1}. \label{uno}
\end{align}
By Lemma \ref{perimetribound} it holds that if fixed $\delta_{0}>0$, if $|\Sigma|<\delta_{0}$ then 
\[
p(t)-2\mathcal H^{n-1}(\Gamma_{1}) \ge  C \mu (t),
\]
where $C$ is the constant given in \eqref{costante}. Hence, substituting in \eqref{uno} we get
\[
\mathcal H^{n-1}(\Gamma_{1}) +\frac C2 \int_{\Omega} u^{2}dx \le
 \frac{\mathcal H^{n-1}(\Gamma_{1})}{2}+ \int_{\Omega} u |\nabla u|\,dx + \frac{1}{2} \int_{\Gamma_{0}} u^{2}\, d\mathcal H^{n-1}.
\]
On the other hand,
\[
\int_{\Omega} u|\nabla u|dx\le \frac{1}{2\eps} \int_{\Omega}u^{2}dx+\frac{\eps}{2}\int_{\Omega} |\nabla u|^{2}dx.
\]
Choosing $\eps=\frac 1\beta$ it holds that
\begin{equation}
\label{quasifine}
 \mathcal H^{n-1}(\Gamma_{1})  + \left(C-\beta\right)\int_{\Omega} u^{2}dx \le  \frac 1\beta  \int_{\Omega}|\nabla u|^{2}dx+ \int_{\Gamma_{0}} u^{2}\,d\mathcal H^{n-1}.
\end{equation}
Therefore, being $R<\frac{n-1}{\beta}$, for $\delta_{0}$ sufficiently small the constant $C$ is larger than $\beta$. Hence the inequality \eqref{quasifine} implies \eqref{fine}.
\end{proof}

\section*{Acnowledgements}
This work has been partially supported by a MIUR-PRIN 2017 grant ``Qualitative and quantitative aspects of nonlinear PDE's'' and by GNAMPA of INdAM.

\end{document}